\documentclass[12pt]{article}
\usepackage{amssymb}
\usepackage{amsmath}
\usepackage{amsthm}
\usepackage{color}
\usepackage{comment}
\usepackage{geometry}		
\usepackage{graphicx}

\makeatletter
	\@addtoreset{equation}{section}
\makeatother

\newcommand\bbbr{\mathbb{R}}
\newcommand\cN{\mathcal{N}}
\newcommand\cX{\mathcal{X}}

\newcommand\myEquals{\hspace{0.9mm}=\hspace{1.1mm}}
\newcommand\myMedSep{\hspace{2mm}}
\newcommand\myBigSep{\hspace{20mm}}
\newcommand{\dfem}[1]{{\bf #1}}
\newcommand{\ee}{\varepsilon}

\let\originalleft\left
\let\originalright\right
\renewcommand{\left}{\mathopen{}\mathclose\bgroup\originalleft}
\renewcommand{\right}{\aftergroup\egroup\originalright}

\newtheorem{theorem}{Theorem}[section]
\newtheorem{lemma}[theorem]{Lemma}

\theoremstyle{definition}
\newtheorem{definition}{Definition}[section]

\begin{document}
\title{How to compute multi-dimensional stable and unstable manifolds of piecewise-linear maps.}
\author{
D.J.W.~Simpson\\\\
School of Mathematical and Computational Sciences\\
Massey University\\
Palmerston North, 4410\\
New Zealand
}
\maketitle

\begin{abstract}

For piecewise-linear maps the stable and unstable manifolds of hyperbolic periodic solutions are themselves piecewise-linear.
Hence compact subsets of these manifolds can be represented using polytopes (i.e.~polygons, in the case of two-dimensional manifolds).
Such representations are efficient and exact
so for computational purposes are superior to representations
that use a large number of points on some mesh (as is usually done in the smooth setting).
We introduce a method for computing convex polytope representations of stable and unstable manifolds.
For an unstable manifold we iterate a suitably small subset of the local unstable manifold
and prior to each iteration subdivide polytopes where they intersect the switching manifold of the map.
We prove the output converges to the (entire) unstable manifold
and use it to visualise attractors and bifurcations of the three-dimensional border-collision normal form:
we identify a heterodimensional-cycle,
a two-dimensional unstable manifold whose closure appears to be a unique attractor,
and a piecewise-linear analogue of a first homoclinic tangency where an attractor appears to be destroyed.

\end{abstract}

\section{Introduction}
\label{sec:intro}

For nonlinear dynamical systems
it is often extremely helpful to understand the global behaviour of
the stable and unstable manifolds of saddle-type invariant sets.
This is because stable manifolds typically form boundaries for basins of attraction
and under parameter variation chaos can be generated and attractors can be destroyed
when stable and unstable manifolds first intersect \cite{GrOt83,PaTa93}.

One-dimensional manifolds are relatively easy to compute.
To compute a one-dimensional unstable manifold of an equilibrium $x^*$ of a system of ordinary differential equations,
one just needs to evolve two points $x^* \pm \ee v$, where $v$ is the associated
unstable eigenvector and $\ee > 0$ is suitably small.
For a map (system of difference equations)
one iterates a large number of points distributed across a fundamental domain \cite{Ku04}.
Computations of two-dimensional manifolds require more effort
and many methods have been developed for doing this \cite{KrOs05}.
However, these are designed for smooth dynamical systems for which the manifolds are curved surfaces.
To represent these computationally it is necessary
to use a large number of points on some two-dimensional mesh.

Here we present a new method for computing multi-dimensional stable and unstable manifolds of piecewise-linear maps.
For simplicity we only treat maps that are continuous and have a single switching manifold,
but the same approach should be effective for discontinuous maps and maps with several switching manifolds.
The method works by repeatedly iterating a suitable initial local approximation $U$,
and a typical output is shown in Fig.~\ref{fig:twoDimInvManExA_many}.
At each step the method generates a compact subset of the manifold
and represents it as a union of convex polytopes.
This approach is based on the fact that the image under the map
of a convex polytope is either another convex polytope
or a union of two convex polytopes.
Computationally each polytope is represented by the convex hull of its vertices.
In this way relatively large subsets of the manifold can be characterised using relatively few points.

\begin{figure}[b!]
\begin{center}
\includegraphics[width=12.2cm]{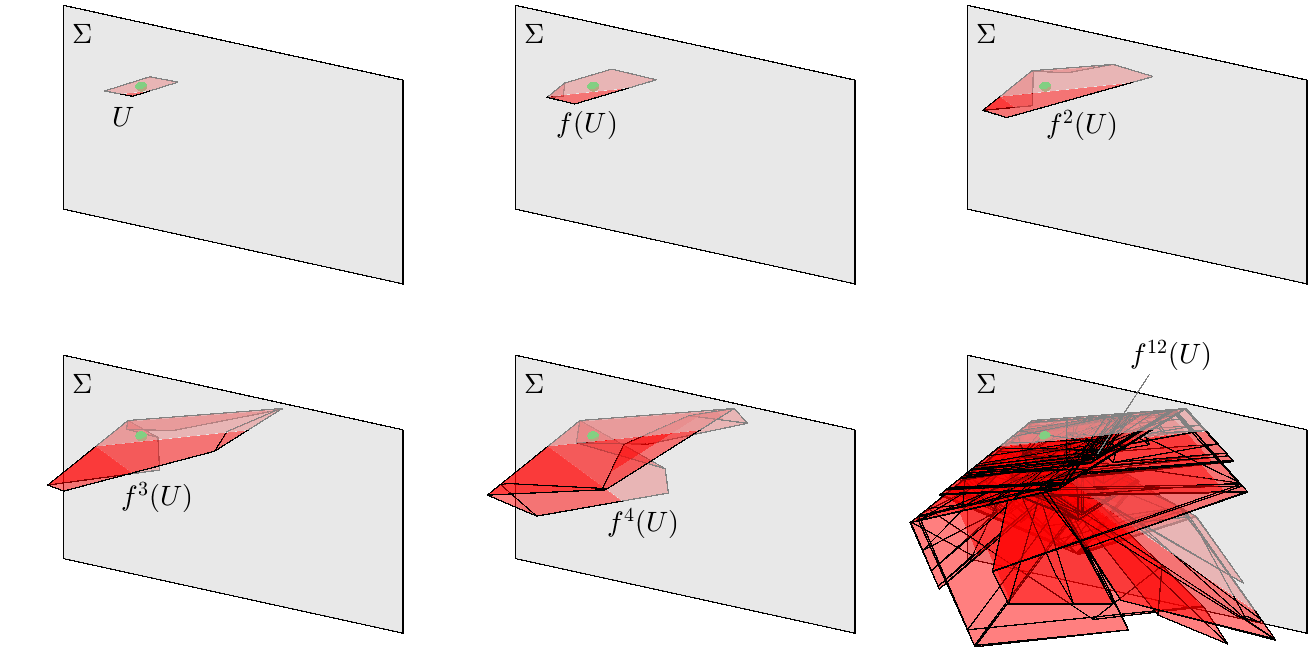}
\caption{
A numerically computed unstable manifold of a saddle fixed point (green dot) with two unstable directions.
This figure is for the three-dimensional border-collision normal form \eqref{eq:f} with \eqref{eq:nf}
using the parameter values \eqref{eq:A}.
Each plot shows the switching manifold $\Sigma$, where the map is continuous but non-differentiable.
The computation starts with a quadrilateral $U$ (a subset of the local unstable manifold),
and after $12$ iterations the computed manifold consists of $533$ polygons (bottom right).
\label{fig:twoDimInvManExA_many}
} 
\end{center}
\end{figure}

Below we prove that the computation limits to the entire manifold as the number of steps tends to infinity
and use the method to compute various two-dimensional manifolds.
Higher dimensional manifolds are not attempted here
as they are significantly more difficult to visualise,
plus require algorithms in computational geometry
to encode the polytopes and split them at the switching manifold \cite{PrSh85}.
Such algorithms are not needed for two-dimensional manifolds
because the vertices of polygons admit a natural cyclical ordering.

This work is motivated by a need to
better understand how dynamics changes at border-collision bifurcations
where a fixed point of a piecewise-smooth map collides with a switching manifold
and the local dynamics is captured by a piecewise-linear approximation \cite{Si16alt}.
Border-collision bifurcations have been heavily studied
for one-dimensional maps \cite{NuYo95,SuAv16}
and two-dimensional maps \cite{BaGr99,NuYo92}
where stable and unstable manifolds of saddles are at most one-dimensional.
Recent studies have revealed novel dynamics in three dimensions \cite{MuBa23}.
To understand these further detailed computations of two-dimensional manifolds should prove useful.
Since the method computes the manifolds extremely accurately (with linearity no approximations are needed),
their computations (even of manifolds that are more than two dimensional)
could be used in computer-assisted proofs \cite{GlSi22b}.

We start in \S\ref{sec:sym} by clarifying the class of maps under consideration
and how periodic solutions can be encoded symbolically.
In \S\ref{sec:polytopes} we review elementary properties of convex polytopes that underpin the computations.
In \S\ref{sec:algorithm} we explain the method in more detail and prove that
the computed subsets converge to the (full) manifold in the limit of infinitely many iterations.
In \S\ref{sec:examples} we discuss some details for how the method can be implemented for two-dimensional manifolds
and illustrate this with the three-dimensional border-collision normal form.
Finally \S\ref{sec:conc} provides a brief discussion.

\section{Periodic solutions and symbolic itineraries}
\label{sec:sym}

Here we discuss the basic aspects of periodic solutions of piecewise-linear maps
that will be needed below.
Further details on this topic can be found in \cite{Si16alt}.

We consider maps on $\bbbr^n$ of the form
\begin{equation}
f(x) = \left\{ \begin{array}{l@{~~}l}
A_L x + b, & c^{\sf T} x \le 0, \\
A_R x + b, & c^{\sf T} x \ge 0,
\end{array} \right.
\label{eq:f}
\end{equation}
where $A_L$ and $A_R$ are $n \times n$ matrices and $b, c \in \bbbr^n$.
The map is assumed to be continuous on the switching manifold $c^{\sf T} x = 0$,
thus $A_L$ and $A_R$ differ by a rank-one matrix, specifically $A_R = A_L + a c^{\sf T}$
for some $a \in \bbbr^n$.
We assume $c$ is not the zero vector so that the switching manifold, call it $\Sigma$,
is a codimension-one manifold (in fact a hyperplane).
If $\det(A_L) \det(A_R) > 0$ the map is invertible,
meaning every $x \in \bbbr^n$ has a unique preimage $f^{-1}(x)$.

To describe orbits symbolically we use the following definition.
For any $x \in \bbbr^n$ not belonging to $\Sigma$, define
$$
\sigma(x) = \left\{ \begin{array}{l@{~~}l} L, & c^{\sf T} x < 0, \\ R, & c^{\sf T} x > 0. \end{array} \right.
$$
In this paper we will not need to assign a symbol to points on $\Sigma$.

Now suppose \eqref{eq:f} has a period-$p$ solution $\gamma$.
For any $y \in \gamma$ we can express $\gamma$ as the ordered set
$\{ y, f(y), \ldots, f^{p-1}(y) \}$.
These points are distinct and $f^p(y) = y$.
Assuming $\gamma$ has no points on $\Sigma$, we can use $y$ to define the word
\begin{equation}
\cX = \sigma(y) \sigma(f(y)) \ldots \sigma(f^{p-1}(y)).
\label{eq:cX}
\end{equation}
That is, $\cX = \cX_0 \cX_1 \cdots \cX_{p-1}$
where $\cX_i = \sigma \left( f^i(y) \right)$ for each $i = 0,1,\ldots,p-1$.
We then refer to $\gamma$ as an $\cX$-cycle.
Different points in $\gamma$ generate different words,
but these words will all be cyclic permutations of one another.
Thus any period-$p$ solution with no points on $\Sigma$
is an $\cX$-cycle for a word $\cX$ of length $p$ that is unique up to cyclic permutation.

Stable and unstable manifolds of periodic solutions are defined as follows.

\begin{definition}
The \dfem{stable manifold} of $\gamma$ is
$$
W^s(\gamma) = \left\{ x \in \bbbr^n \,\big|\, f^i(x) \to \gamma {\rm ~as~} i \to \infty \right\},
$$
and the \dfem{unstable manifold} of $\gamma$ is
$$
W^u(\gamma) = \left\{ x \in \bbbr^n \,\big|\, x {\rm ~has~a~sequence~of~preimages~converging~to~} \gamma \right\}.
$$
\label{df:WsWu}
\end{definition}

Continuing to assume $\gamma$ has no points on $\Sigma$,
each point of $\gamma$ has a neighbourhood in which $f^p$ is smooth.
Thus we can use classical dynamical systems theory
to help us understand the nature of its stable and unstable manifolds.

In fact in this neighbourhood $f^p$ is affine and can be expressed explicitly as follows.
By composing the pieces of $f$ in the order specified by the word $\cX$,
in a neighbourhood of $y$,
\begin{equation}
f^p(x) = M_\cX x + P_\cX b,
\label{eq:fp}
\end{equation}
where
\begin{eqnarray}
M_\cX &=& A_{\cX_{p-1}} \cdots A_{\cX_1} A_{\cX_0} \,, \label{eq:MX} \\
P_\cX &=& I + A_{\cX_{p-1}} + A_{\cX_{p-1}} A_{\cX_{p-2}} + \cdots + A_{\cX_{p-1}} A_{\cX_{p-2}} \cdots A_{\cX_1} \,. \label{eq:PX}
\end{eqnarray}
In this neighbourhood ${\rm D} f^p(x) = M_\cX$,
thus the stability multipliers of $\gamma$ are the eigenvalues of $M_{\cX}$.
Notice the same eigenvalues result using any other point in $\gamma$ 
because these eigenvalues are independent of the cyclical ordering of the $p$ matrices in \eqref{eq:MX}.

If none of the eigenvalues has unit modulus, as is generically the case, then $\gamma$ is hyperbolic.
In this case we can apply the local stable manifold theorem \cite{Ro99}.
Any point in $\gamma$, say $y$, is a fixed point of $f^p$,
and its stable and unstable subspaces are the invariant subspaces of $x \mapsto M_\cX x$
that are aligned with all stable and unstable directions, respectively.
The local stable manifold theorem ensures
$y$ has local stable and unstable manifolds of the same dimensions as the corresponding subspaces
and tangent to these subspaces at $y$.
Further, $W^u(\gamma)$ can be constructed by iterating the local unstable manifold under $f$,
and if $f$ is invertible $W^s(\gamma)$ can be constructed by iterating the local stable manifold under $f^{-1}$.

But in our piecewise-linear setting
the local stable and unstable manifolds {\em coincide} with the appropriately translated stable and unstable subspaces.
For this reason the small set $U$ that we will use to initialise our computation
is not an approximation to a small part of the stable or unstable manifold,
it {\em is} part of the manifold.

\section{Fundamentals of convex polytopes}
\label{sec:polytopes}

Here we clarify the definition of a convex polytope
and the concepts of dimension and relative interior that will be needed below.
For a more gradual and detailed introduction to these topics refer to
the first three sections of Br{\o}ndsted \cite{Br83}.

A set $C \subset \bbbr^n$ is {\em convex}
if $(1-s) x + s y \in C$ for all $x, y \in C$ and $0 \le s \le 1$.
The {\em convex hull} of a set $S \subset \bbbr^n$ is
the smallest convex set containing $S$.

\begin{definition}
A set $P \subset \bbbr^n$ is a \dfem{convex polytope}
if it is the convex hull of a non-empty finite set.
\label{df:convexPolytope}
\end{definition}

Fig.~\ref{fig:twoDimInvManConvexPolytope} shows an convex polytope in $\bbbr^3$.
Here $P$ is the convex hull of four points (its vertices).
These points lie on a plane, so $P$ is two-dimensional.
Formally the dimension of a convex polytope
can be defined as follows.

\begin{figure}[b!]
\begin{center}
\includegraphics[width=7.2cm]{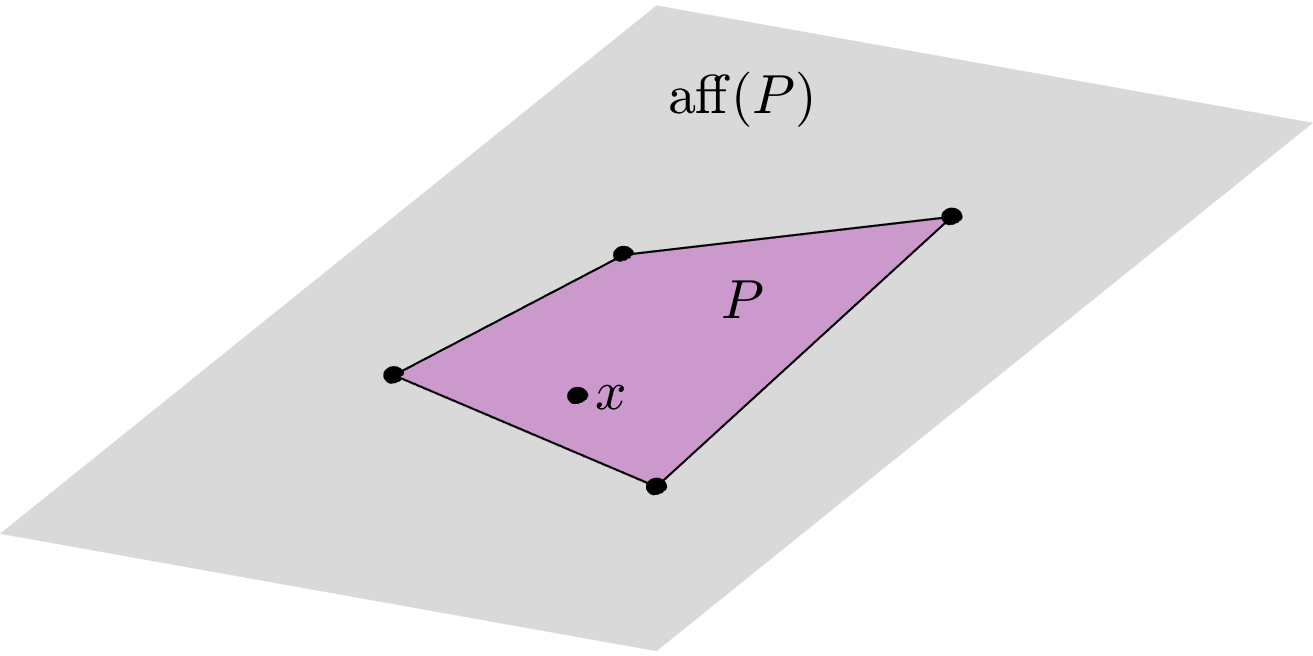}
\caption{
A sketch of a convex polygon $P \subset \bbbr^3$, its affine hull ${\rm aff}(P)$,
and a point $x$ that belongs to the relative interior of $P$.
\label{fig:twoDimInvManConvexPolytope}
} 
\end{center}
\end{figure}

A set $V \subset \bbbr^n$ is
an {\em affine subspace} it is a translate of a (linear) subspace.
The {\em affine hull} of a set $S \subset \bbbr^n$ is
the smallest affine subspace containing $S$.

\begin{definition}
The \dfem{dimension} of a convex polytope is the dimension of its affine hull
(which is the dimension of its linear translate).
\label{df:dimension}
\end{definition}

Below we use the concept of relative interior instead of interior
because if the dimension of a convex polytope $P \subset \bbbr^n$
is less than $n$, then its interior (with respect to $\bbbr^n$) is the empty set.

\begin{definition}
The \dfem{relative interior} of a convex polytope
is its interior taken with respect to its affine hull.
\label{df:relativeInterior}
\end{definition}

In the case of polygons (two-dimensional polytopes),
a point in the polygon belongs to its relative interior
if and only if it is not a vertex or lies on an edge of the polygon,
see again Fig.~\ref{fig:twoDimInvManConvexPolytope}.

The following result is readily proved by direct means \cite{Br83}.

\begin{lemma}
Let $P$ be the convex hull of $x_1,\ldots,x_K \in \bbbr^n$
and let $g : \bbbr^n \to \bbbr^n$ be an affine map.
Then $g(P)$ is the convex hull of $g(x_1),\ldots,g(x_K)$.
\label{le:imageOfPolytope}
\end{lemma}

Next we consider polytopes in the phase space of a piecewise-linear map \eqref{eq:f}.
Recall, $\Sigma$ denotes the switching manifold $c^{\sf T} x = 0$.	

\begin{definition}
A connected set $S \subset \bbbr^n$ \dfem{crosses} $\Sigma$
if there exist $x_1, x_2 \in S$ with $c^{\sf T} x_1 < 0$ and $c^{\sf T} x_2 > 0$.
\label{df:crosses}
\end{definition}

If a convex polytope crosses $\Sigma$,
then $\Sigma$ divides the polytope into two sets, Fig.~\ref{fig:twoDimInvManIntersection}.
These sets are themselves convex polytopes.
This is a trivial consequence of the exterior representation of a convex polytope:
any convex polytope can be expressed as the convex hull of its vertices (the interior representation)
or as the intersection of a finite set of half-spaces (the exterior representation).
Thus $\Sigma$ merely introduces an additional half-space; also the dimension is unchanged:

\begin{lemma}
A $d$-dimensional convex polytope that crosses $\Sigma$
is the union of two $d$-dimensional convex polytopes that do not cross $\Sigma$.
\label{le:union}
\end{lemma}

\begin{figure}[h!]
\begin{center}
\includegraphics[width=4.8cm]{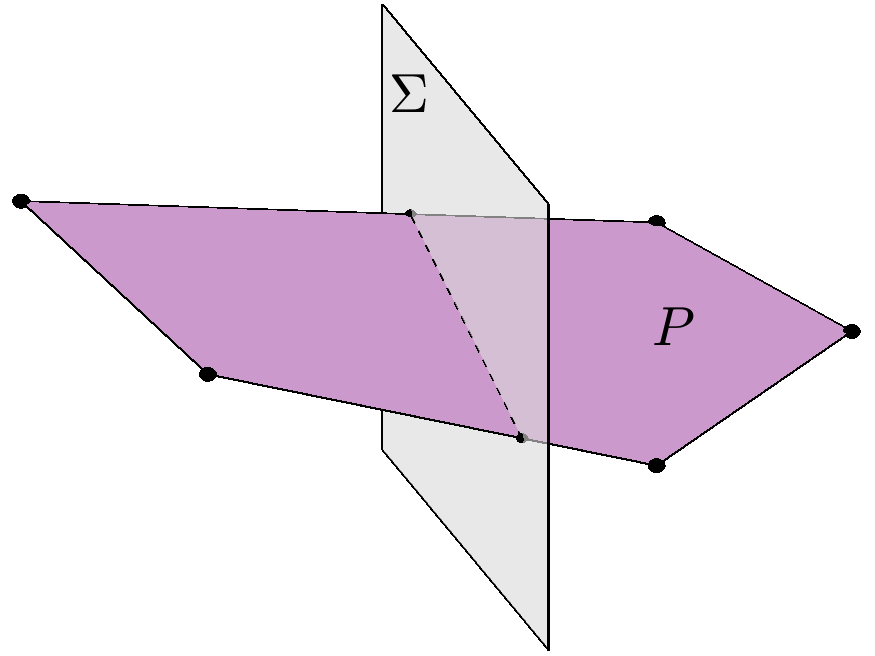}
\caption{
A sketch of the intersection of a convex polygon $P \subset \bbbr^3$
with the switching manifold $\Sigma$.
\label{fig:twoDimInvManIntersection}
} 
\end{center}
\end{figure}

\section{Computing unstable manifolds}
\label{sec:algorithm}

Here we explain our method for computing multi-dimensional unstable manifolds of a piecewise-linear map $f$ \eqref{eq:f}.
If the map is invertible then stable manifolds can be computed in the same way
by using $f^{-1}$ in place of $f$.

Let $\gamma$ be a period-$p$ solution with no points on $\Sigma$.
Assume $\gamma$ is hyperbolic, let $y \in \gamma$, and let $\cX$ be the word \eqref{eq:cX}.
Let $H$ be the affine subspace that contains $y$,
has the same dimension as the unstable subspace of $M_\cX$, and is tangent to this subspace.
As discussed in \S\ref{sec:sym}, the branch of $W^u(\gamma)$
that emanates from $y$ does so coincident with $H$.

Let $d$ be the dimension of $H$ (this is the unstable index of $\gamma$), and assume $1 \le d \le n-1$.
Let $U \subset H$ be a $d$-dimensional convex polytope that contains $y$ in its relative interior
and obeys the following {\em admissibility} condition:
for every vertex $v$ of $U$
there exists a sequence of preimages $\left\{ f^{-i}(v) \right\}_{i=1}^\infty$ satisfying
\begin{equation}
\sigma \left( f^{-i}(v) \right) = \cX_{-i \,{\rm mod}\, p} \,, \qquad {\rm for~all~} i \ge 1.
\label{eq:backwardsAdmissibility}
\end{equation}
This condition ensures that $U$ is not too big.
It says that each vertex of $U$ has preimages repeatedly the following the symbols in $\cX$ (in reverse order).
Since $H$ contains only unstable directions, these preimages converge to $\gamma$. 
The convexity of $U$ then gives the following result (proved in Appendix \ref{app:app}).

\begin{lemma}
With the above assumptions, every point in $U$ belongs to $W^u(\gamma)$.
\label{le:U}
\end{lemma}

Now let $Q^{(0)} = P^{(0)}_1 = U$, and for all $i \ge 1$ define
$$
Q^{(i)} = f \big( Q^{(i-1)} \big).
$$
Computationally we represent each $Q^{(i)}$ as a union of convex polytopes:
\begin{equation}
Q^{(i)} = \bigcup_{j=1}^{m_i} P^{(i)}_j.
\label{eq:Qi}
\end{equation}
To do this we use the images of the polygons in $Q^{(i-1)}$ to form $Q^{(i)}$:
by Lemmas \ref{le:imageOfPolytope} and \ref{le:union}
the image under $f$ of any convex polytope
is either another convex polytope or the union of two convex polytopes.

Observe $Q^{(i)}$ belongs to the branch of $W^u(\gamma)$ that contains
$f^{i \,{\rm mod}\, p}(y) \in \gamma$.
Thus
$$
Z^{(r)} = Q^{(r)} \cup Q^{(r+1)} \cup Q^{(r+p-1)},
$$
where $r \ge 0$, includes exactly one part of each branch of $W^u(\gamma)$.
The following result justifies using
$Z^{(r)}$ with a large value of $r$ to approximate $W^u(\gamma)$.

\begin{theorem}
For any $x \in \bbbr^n$ the following are equivalent:
\begin{enumerate}
\item
$x \in W^u(\gamma)$;
\item
there exists $N \ge 0$ such that $x \in Z^{(r)}$ for all $r \ge N$.
\end{enumerate}
\label{th:main}
\end{theorem}

\begin{proof}
In a neighbourhood $\cN$ of $y$ where $f^p$ is affine, $H$ is locally invariant under $f^p$.
Since $U \subset H$ has $y$ in its relative interior
and $H$ contains only unstable directions,
no point on the boundary of $U \cap \cN$ (taken with respect to $H$) converges to $y$ in $\cN$ under $f^p$.
Thus there exists a $d$-dimensional subset $V \subset U$ with $y$ in its relative interior
such that $V \subset f^{\ell p}(U)$ for all $\ell \ge 0$.
Hence $V \subset Q^{(\ell p)}$ for all $\ell \ge 0$,
so $V \subset Z^{(r)}$ for all $r \ge 0$.

Now suppose $x \in W^u(\gamma)$.
Then $x$ has a sequence of preimages under $f$ that converges to $\gamma$.
Since $V \subset H$ has $y$ in its relative interior and contains all unstable directions,
this sequence must contain a point $w \in V$.
Let $N \ge 0$ be such that $x = f^N(w)$.
Then $w \in Z^{(r)}$ for all $r \ge 0$,
thus $x \in Z^{(r)}$ for all $r \ge N$.

Conversely suppose there exists $N \ge 0$ such that $x \in Z^{(r)}$ for all $r \ge N$.
Let $i \in \{ 0,1,\ldots,p-1 \}$ be such that $x \in Q^{(N+i)}$,
and $w \in U$ be such that $f^{N+i}(w) = x$.
Then $w \in W^u(\gamma)$ by Lemma \ref{le:U}, thus $x \in W^u(\gamma)$ as required.
\end{proof}

\section{Two-dimensional implementation and examples}
\label{sec:examples}

In this section we first give details of
the implementation of the method for two-dimensional unstable manifolds,
then describe three examples.

To initialise the method we need a suitable set $U$.
To construct $U$ we first evaluate $y = \left( I - M_{\cX} \right)^{-1} P_{\cX} b$ (the unique fixed point of \eqref{eq:fp}),
and use the eigenvectors of $M_{\cX}$ to identify a basis $\{ u_1, u_2 \}$ of the unstable subspace.
Any point in $H$ can then be written as $y + k_1 u_1 + k_2 u_2$, for some $k_1, k_2 \in \bbbr$.

Recall $H$ is invariant under \eqref{eq:fp}.
Thus the restriction of \eqref{eq:fp} to $H$ can be expressed as an invertible
two-dimensional map on the values of $k_1$ and $k_2$.
This map, call it $g$, will be needed in a moment.

Let $K \ge 3$ be the number of vertices we want $U$ to have.
To ensure $y$ lies in the relative interior of $U$,
we define $K$ pairs of points $(k_1,k_2)$ equispaced on a circle in $\bbbr^2$ centred at the origin,
and use these to construct $K$ points $y + k_1 u_1 + k_2 u_2$ to be the vertices of $U$.
These points will satisfy the admissibility condition \eqref{eq:backwardsAdmissibility}
if the radius of the circle is sufficiently small.
Numerically \eqref{eq:backwardsAdmissibility} can be verified for all $i = 1,2,\ldots,1000$, say
(in which case it is almost certainly true for all $i \ge 1$), by explicitly computing preimages.
To compute the preimages it is important to use $g^{-1}$
because $\gamma$ is a saddle thus numerical (round-off) error
will mean that preimages of $f$ will not converge to $\gamma$.

Once an appropriate set $U$ has been identified,
we store $P^{(0)}_1 = U$ as a list of vertices, and also call this $Q^{(0)}$.
By \eqref{eq:Qi}, each $Q^{(i)}$ is a union of $m_i$ convex polygons $P^{(i)}_j$, each of which is stored as a list of vertices.
To compute each $Q^{(i+1)}$ we perform the following steps to each $P^{(i)}_j$.
First we evaluate $c^{\sf T} v$ at every vertex $v$ of $P^{(i)}_j$.
If no two of these values have different signs then $P^{(i)}_j$ does not cross $\Sigma$ by convexity.
Otherwise $P^{(i)}_j$ does cross $\Sigma$, in which case its edges intersect $\Sigma$ at exactly two points.
It is a simple coding exercise to compute these points then use them and the vertices of $P^{(i)}_j$
to form two convex polygons that do not cross $\Sigma$ and whose union is $P^{(i)}_j$.
In either case we then map every vertex under $f$
to create one or two polygons in $Q^{(i+1)}$ (this step is justified by Lemma \ref{le:imageOfPolytope}).

We now show computed manifolds of the three-dimensional border-collision normal form.
This is the map \eqref{eq:f} on $\bbbr^3$ with
\begin{equation}
A_L = \left[ \begin{array}{c@{\myMedSep}c@{\myMedSep}c}
\tau_L & 1 & 0 \\
-\sigma_L & 0 & 1 \\
\delta_L & 0 & 0
\end{array} \right], \quad
A_R = \left[ \begin{array}{c@{\myMedSep}c@{\myMedSep}c}
\tau_R & 1 & 0 \\
-\sigma_R & 0 & 1 \\
\delta_R & 0 & 0
\end{array} \right], \quad
b = \left[ \begin{array}{c}
1 \\
0 \\
0
\end{array} \right], \quad
c = \left[ \begin{array}{c}
1 \\
0 \\
0
\end{array} \right],
\label{eq:nf}
\end{equation}
and is a normal form in the sense that any three-dimensional map of the form \eqref{eq:f}
that satisfies a certain genericity condition (observability)
can be transformed to the normal form under an affine change of coordinates \cite{Di03}.
Studies of bifurcations in the three-dimensional normal form include \cite{DeDu11,MuBa23,Pa18,Si17c}.

\begin{figure}[b!]
\begin{center}
\includegraphics[width=12.2cm]{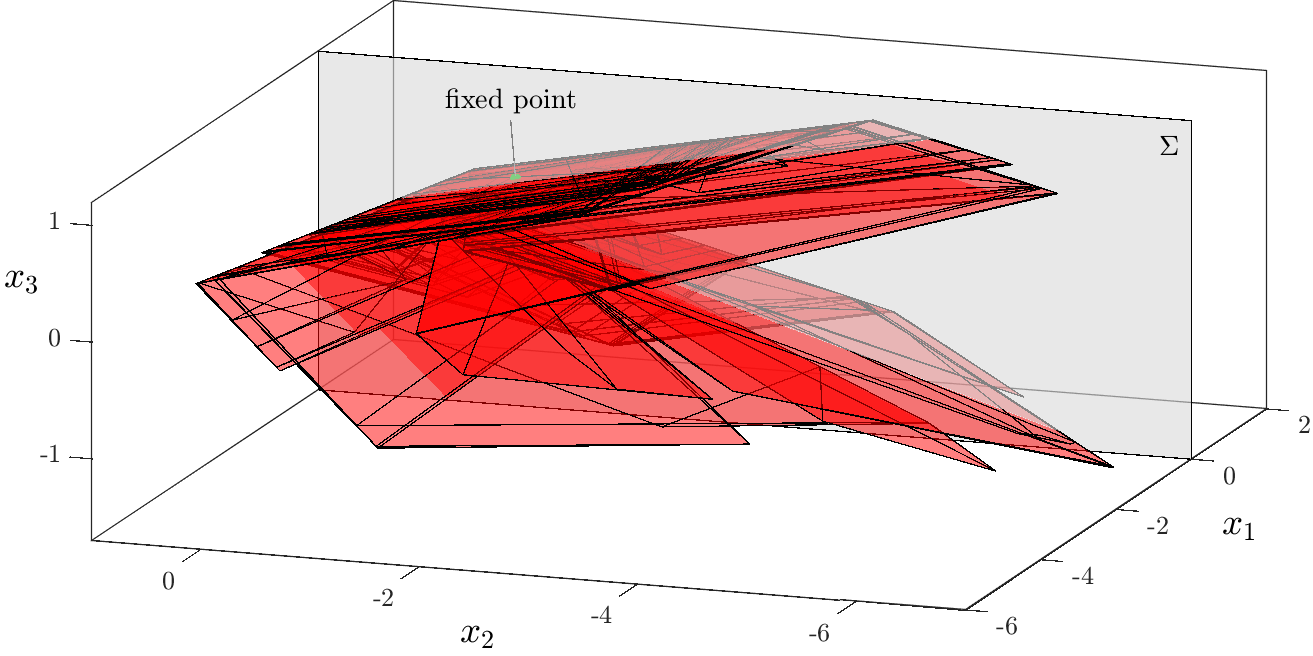}
\caption{
A plot of the phase space of the three-dimensional border-collsion normal form with parameter values \eqref{eq:A}
showing a two-dimensional unstable manifold whose closure appears to be a chaotic attractor.
This is an enlarged view of the bottom right plot of Fig.~\ref{fig:twoDimInvManExA_many}.
\label{fig:twoDimInvManExA}
} 
\end{center}
\end{figure}

First we consider the normal form with parameter values
\begin{equation}
\begin{array}{r@{\myEquals}l@{\myBigSep}r@{\myEquals}l}
\tau_L & 0, & \tau_R & 0, \\
\sigma_L & -1, & \sigma_R & 3, \\
\delta_L & 0.3, & \delta_R & 0.6.
\end{array}
\label{eq:A}
\end{equation}
These values were obtained by starting with Figure 1 of \cite{Gl16e} for an instance of the two-dimensional normal form
(which corresponds to $\delta_L = \delta_R = 0$ in \eqref{eq:nf}) for which the map has a two-dimensional attractor,
and altering the parameter values slightly to create fully three-dimensional dynamics.
Numerical simulations (not shown) suggest that with \eqref{eq:A} the normal form
has a chaotic	attractor that is equal to the closure of the unstable manifold of a fixed point.
This manifold is two-dimensional and its computation is shown in Fig.~\ref{fig:twoDimInvManExA}
using $i = 12$ iterations of the quadrilateral $U$ shown in Fig.~\ref{fig:twoDimInvManExA_many}.
The computation gives $f^{12}(U)$ as the union of $533$ polygons.
These are plotted as semi-transparent surfaces
and provides some tangible impression of the presumably fractal geometry of the attractor.

Next we use the values
\begin{equation}
\begin{array}{r@{\myEquals}l@{\myBigSep}r@{\myEquals}l}
\tau_L & 1.5, & \tau_R & 0, \\
\sigma_L & 0, & \sigma_R & 1.5, \\
\delta_L & 0.5, & \delta_R & 0.5,
\end{array}
\label{eq:B}
\end{equation}
obtained by altering values used in Figure 8 of \cite{GlSi22b}.
With \eqref{eq:B} the normal form has a saddle fixed point with a one-dimensional unstable manifold and a two-dimensional stable manifold.
From their numerical computation, Fig.~\ref{fig:twoDimInvManExB},
we find that these manifolds nearly intersect at points far from the fixed point.
For instance, near the middle of the figure the stable manifold (blue) has `spikes'
close to the unstable manifold (red).
This suggests that the values \eqref{eq:B} are close to where the stable and unstable manifolds first intersect
(indeed with instead $\tau_L = 1.51$ the right branch of the unstable manifold diverges).
Such an intersection is a piecewise-linear analogue of a first homoclinic tangency.
Numerical simulations suggest that an attractor
is destroyed when the parameter values are perturbed to create such an intersection.

\begin{figure}[b!]
\begin{center}
\includegraphics[width=12.2cm]{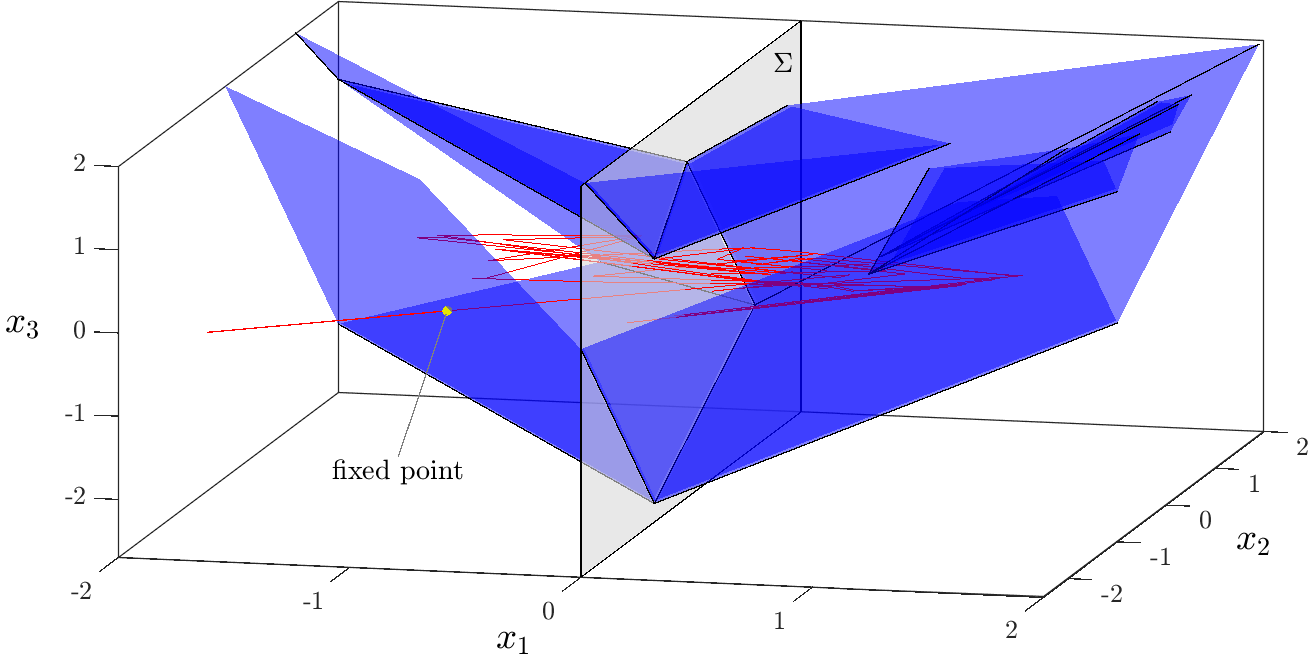}
\caption{
The two-dimensional stable (blue) and one-dimensional unstable (red)
manifolds of a fixed point of the three-dimensional border-collision normal form with parameter values \eqref{eq:B}.
The stable manifold appears to form the boundary of the basin of attraction of a chaotic attractor
that, when parameters are varied (e.g.~$\tau_L$ is increased slightly),
is destroyed due to the stable and unstable manifolds developing non-trivial intersections.
\label{fig:twoDimInvManExB}
} 
\end{center}
\end{figure}

\begin{figure}[b!]
\begin{center}
\includegraphics[width=12.2cm]{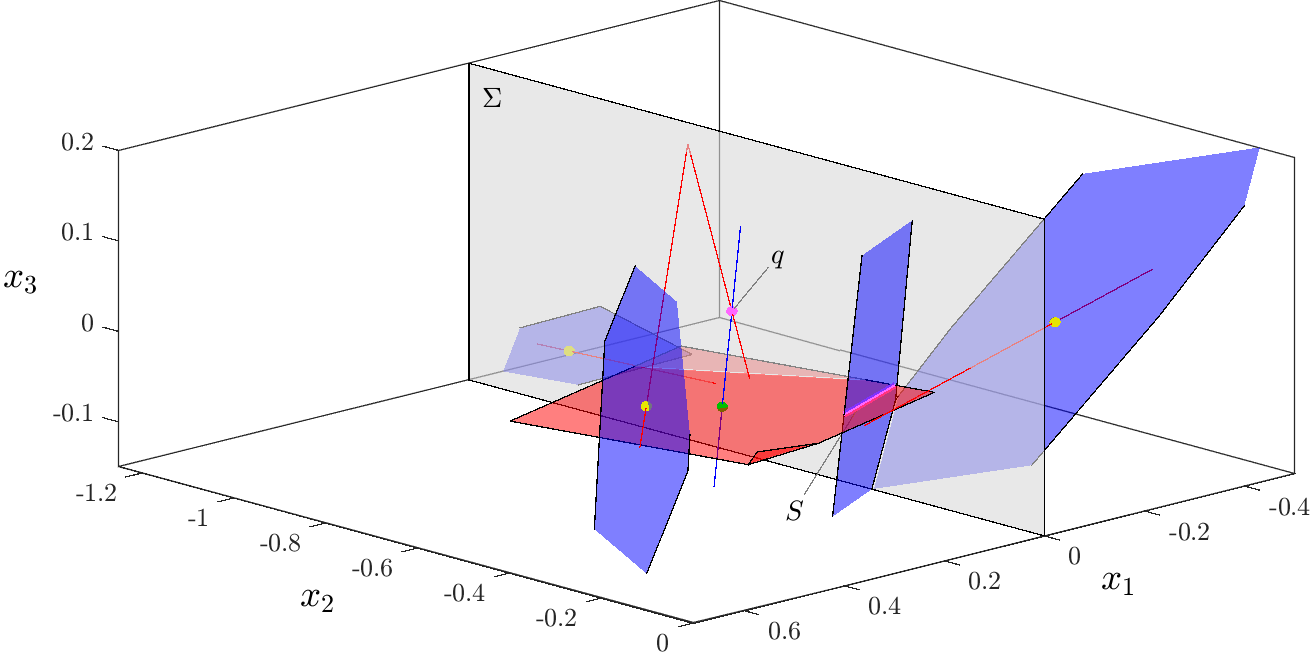}
\caption{
A plot of the phase space of the three-dimensional border-collsion normal form with parameter values \eqref{eq:C}
showing the stable (blue) and unstable (red) manifolds of a fixed point and $LLR$-cycle.
Their respective stable and unstable manifolds intersect at $q$ and along $S$,
hence these sets have a heterodimensional-cycle.
\label{fig:twoDimInvManExC}
} 
\end{center}
\end{figure}

Finally we consider the values
\begin{equation}
\begin{array}{r@{\myEquals}l@{\myBigSep}r@{\myEquals}l}
\tau_L & 0.7228540306, & \tau_R & -1.5, \\
\sigma_L & -1, & \sigma_R & 2, \\
\delta_L & -0.2, & \delta_R & -0.2,
\end{array}
\label{eq:C}
\end{equation}
used for Figure 6 of \cite{GlSi23b}.
Here the normal form has a saddle fixed point (the green point in Fig.~\ref{fig:twoDimInvManExC}) with unstable index two,
and a saddle $LLR$-cycle (the yellow points in Fig.~\ref{fig:twoDimInvManExC}) with unstable index one.
The value of $\tau_L$ was chosen in \cite{GlSi23b} (accurate to ten decimal places) so that
the one-dimensional stable manifold of the fixed point intersects the one-dimensional unstable manifold of the $LLR$-cycle
(e.g.~at the point $q$).
The two-dimensional manifolds were not computed in \cite{GlSi23b};
they have been computed here and found to intersect (e.g.~along the line segment $S$).
For clarity the manifolds have only been grown as far as needed to identify these intersections.

Together the intersections show that there exists a heteroclinic connection from the fixed point to the $LLR$-cycle,
and another connection from the $LLR$-cycle back to the fixed point.
Significantly, the fixed point and $LLR$-cycle have different unstable indices.
Consequently the connections form a {\em heterodimensional cycle};
such cycles are well known for generating non-hyperbolic dynamics \cite{BoDi05}.

\section{Discussion}
\label{sec:conc}

This paper has introduced a method for efficiently computing
multi-dimensional stable and unstable manifolds of piecewise-linear maps.
The method simply iterates polytopes, and it remains to see if
variations on this approach can produce improved performance.
For instance, the method does not explicitly manage cases
where the dynamics on the manifold expands relatively strongly in one direction
that would cause the computation to produce a subset of the manifold that is stretched greatly in one direction.
For stable and unstable manifolds of equilibria of ordinary differential equations this is a common difficulty
that can be dealt with, for instance, by evolving the boundary of the computed subset
so that at each step all points on the boundary have the same geodesic distance
from the equilibrium \cite{KrOs03}.
For polytopes we could impose a similar constraint,
e.g.~that the geodesic distance of all vertices on the boundary of the subset
varies by at most, say, $50\%$.
Also it remains to modify the method to compute stable manifolds when \eqref{eq:f} is non-invertible
by computing all preimages of the subset at each step.

\section*{Acknowledgements}

The author thanks the organisers of ICDEA 2023 where many of the ideas were developed,
and encouragement from Soumitro Banerjee.
This work was supported by Marsden Fund contract MAU2209 managed by Royal Society Te Ap\={a}rangi.

\appendix

\section{Proof of Lemma \ref{le:U}}
\label{app:app}

Let $v^{(1)}, v^{(2)}, \ldots v^{(K)}$ denote the vertices of $U$.
For each $k \in \{ 1,2,\ldots,K \}$ let $v^{(k)}_0 = v^{(k)}$ and let $\left\{ v^{(k)}_i \right\}_{i=1}^\infty$
be a sequence of preimages of $v^{(k)}$ satisfying \eqref{eq:backwardsAdmissibility}.
That is, for all $i \ge 1$,
$$
f \left( v^{(k)}_i \right) = A^{(i)} v^{(k)}_i + b = v^{(k)}_{i-1} \,,
$$
where $A^{(i)} = A_{\cX_{-i \,{\rm mod}\, p}}$.

Choose any $x \in U$ and form the convex combination $x = \sum_{k=1}^K \lambda_k v^{(k)}$.
Let $x_0 = x$ and for each $i \ge 1$ define
$$
x_i = \sum_{k=1}^K \lambda_k v^{(k)}_i.
$$
Notice $x_i \to \gamma$ because $v^{(k)}_i \to \gamma$ for each $k$.
To complete the proof it remains to show that $\{ x_i \}_{i=1}^\infty$ is a sequence of preimages of $x$,
i.e.~$f(x_i) = x_{i-1}$ for all $i \ge 1$.

Choose any $i \ge 1$.
By convexity, $c^{\sf T} x_i = \sum_{k=1}^K \lambda_k c^{\sf T} v^{(k)}_i$ has same sign as each $c^{\sf T} v^{(k)}_i$.
Thus
\begin{eqnarray*}
f(x_i)
&=& A^{(i)} x_i + b
= A^{(i)} \left( \sum_{k=1}^K \lambda_k v^{(k)}_i \right) + b \\
&=& \lambda_k \sum_{k=1}^K \left( A^{(i)} v^{(k)}_i + b \right)
= \sum_{k=1}^K \lambda_k v^{(k)}_{i-1}
= x_{i-1} \,,
\end{eqnarray*}
as required. \hfill \qed



\end{document}